\documentclass[12pt]{article}
\usepackage{amsmath,amssymb,amsthm,amsfonts,amscd,latexsym,verbatim}

\usepackage{indentfirst}

\usepackage[cp1251]{inputenc}

\usepackage{geometry}
\usepackage{microtype}

\usepackage{hyperref}

\let\le\leqslant
\let\ge\geqslant

\let\mc\mathcal

\newcommand{\NN}{\mathbb Z_{>0}}
\newcommand{\NNN}{\mathbb Z_{\ge0}}
\newcommand{\CC}{\mathbb C}

\newcommand{\ZZ}{\mathbb Z}
\newcommand{\QQ}{\mathbb Q}

\newcommand{\eps}{\varepsilon}
\renewcommand{\phi}{\varphi}

\newcommand{\sumjks}{\sum_{j=1}^m\sum_{k=0}^{d-1}\sum_{\sigma=0}^{s_j-1}}

\theoremstyle{plain}
\newtheorem{theorem}{Theorem}
\newtheorem{lemma}{Lemma}

\newtheorem{proposition}{Proposition}
\theoremstyle{remark}

\theoremstyle{definition}

\numberwithin{equation}{section}

\begin{document}

\hypersetup{pdfauthor={igor rochev},%
pdftitle={linear independence measures for values of q-hypergeometric series}}

\title{New linear independence measures for values of $q$-hypergeometric series}

\author{I.~Rochev\thanks{Research is supported by RFBR (grant~No.~09-01-00371a).}}

\date{}

\maketitle

\section{Introduction}

Let $q=q_1/q_2\in\QQ$, where $q_1,q_2\in\ZZ\setminus\{0\}$, $\gcd(q_1,q_2)=1$, $|q_1|>|q_2|$. Put
\begin{equation}\label{eq-gamma}
\gamma=\frac{\log|q_2|}{\log|q_1|}.
\end{equation}
Let $P(z)\in\QQ[z]$ with $d:=\deg P\ge1$. Assume that $P(q^n)\ne0$ for all~$n\in\NN$. Consider the function
\begin{equation*}
f(z)=\sum_{n=0}^\infty\frac{z^n}{\prod_{k=1}^nP(q^k)}.
\end{equation*}

In this note we prove the following theorem.

\begin{theorem}\label{th-1}
Let $\alpha_1,\ldots,\alpha_m\in\QQ^*$ be such that the following conditions hold:
\begin{enumerate}
\item\label{condition-1} $\alpha_j\alpha_k^{-1}\notin q^{\ZZ}$ for all~$j\ne k$,
\item\label{condition-2} $\alpha_j\notin P(0)q^{\NN}$ for all~$j$.
\end{enumerate}
Let $s_1,\ldots,s_m\in\NN$. Put
\begin{gather}
S=s_1+\ldots+s_m,\label{eq-S}\\
M=\begin{cases}
dS+1/2+\sqrt{d^2S^2+1/4},&\text{$P(z)=p_dz^d$, $p_d\in\QQ^*$,}\\
dS+1+\sqrt{dS(dS+1)}&\text{otherwise.}
\end{cases}\label{eq-M}
\end{gather}
Suppose that
\begin{equation*}
\gamma<\frac1M,
\end{equation*}
where $\gamma$ is given by~\eqref{eq-gamma}; then the numbers
\begin{equation*}
1,f^{(\sigma)}(\alpha_jq^k)\qquad(1\le j\le m,0\le k<d,0\le\sigma<s_j)
\end{equation*}
are linearly independent over~$\QQ$. Moreover, there exists a positive constant~$C_0=C_0(q,P,m,\alpha_j,s_j)$ such that for any vector~$\vec A=(A_0,A_{j,k,\sigma})\in\ZZ^{1+dS}\setminus\{\vec0\}$ we have
\begin{equation*}
\left|A_0+\sumjks A_{j,k,\sigma}f^{(\sigma)}(\alpha_jq^k)\right|\ge H^{-\mu-C_0/\sqrt{\log H}},
\end{equation*}
where $H=\max\left\{\max_{j,k,\sigma}|A_{j,k,\sigma}|,2\right\}$,
\begin{equation}\label{eq-mu}
\mu=\frac{M-1}{1-M\gamma}.
\end{equation}
\end{theorem}

The case when all roots of~$P$ are rational and $P(0)=0$ was proved in~\cite{Katsurada-1989} with a larger value for the quantity~\eqref{eq-M} if~$P(z)\ne p_dz^d$ (see also~\cite{Sankilampi-Vaananen-2007}). The qualitative part of the general case for~$q\in\ZZ$ was essentially proved in~\cite{Amou-Vaananen-2005}, where it was assumed that $\alpha_j\notin P(0)q^{\ZZ}$ for all~$j$.

Recently the author~\cite{Rochev-2010} proved quantitative results in the general case under a milder condition posed on~$q$ but with the estimate of the form~$\exp\left(-C(\log H)^{3/2}\right)$, $C=\mathrm{const}$. We modify the method of~\cite{Rochev-2010} to prove Theorem~\ref{th-1}.

\section{Construction of auxiliary linear forms}

Fix $\alpha_1,\ldots,\alpha_m\in\CC^*$, $s_1,\ldots,s_m\in\NN$. By $\vec x$ denote the vector of variables~$\vec x=(x_0,x_{j,k,\sigma})$, where $1\le j\le m$, $0\le k<d$, $0\le\sigma<s_j$.

Consider the sequences of linear forms
\begin{gather}
u_n=u_n(\vec x)=\sumjks\sigma!\binom n\sigma\left(\alpha_jq^k\right)^{n-\sigma}x_{j,k,\sigma}\in\CC[\vec x]\qquad(n\in\ZZ),\label{eq-un}\\
\begin{split}
v_n=v_n(\vec x)&=\prod_{k=1}^nP(q^k)\cdot\left(x_0+\sum_{l=0}^n\frac{u_l(\vec x)}{\prod_{k=1}^lP(q^k)}\right)=\\
&=x_0\prod_{k=1}^nP(q^k)+\sum_{l=0}^nu_l(\vec x)\prod_{k=l+1}^nP(q^k)\in\CC[\vec x]\qquad(n\in\NNN).
\end{split}\label{eq-vn}
\end{gather}
It's readily seen that
\begin{equation}\label{eq-vn-linear-recurrence}
v_n=P(q^n)v_{n-1}+u_n\qquad(n\ge1)
\end{equation}
with $v_0=x_0+u_0=x_0+\sumjks x_{j,k,\sigma}$.

Further, let~$\mc B$ be the backward shift operator given by
\begin{equation*}
\mc B\bigl(\xi(n)\bigr)=\xi(n-1).
\end{equation*}
For $a\in\CC$ introduce the difference operator
\begin{equation}\label{eq-mcD}
\mc D_a=\mc I-a\mc B,
\end{equation}
where $\mc I$ is the identity operator, $\mc I\bigl(\xi(n)\bigr)=\xi(n)$. Note that these operators commute with each other. For example, we have
\begin{equation*}
\mc B\left(\mc D_a\bigl(\xi(n)\bigr)\right)=\mc D_a\bigl(\xi(n-1)\bigr).
\end{equation*}
It's well known that for $a\in\CC^*$ and $p(z)\in\CC[z]$ with~$\deg p\le t\in\NNN$ we have
\begin{equation}\label{eq-linear-recurrence}
\mc D_a^{t+1}\bigl(p(n)a^n\bigr)=0\qquad(n\in\ZZ).
\end{equation}
Also, it is readily seen that for $a,b\in\CC$ with $b\ne0$ we have
\begin{equation}\label{eq-mcD-identity}
\mc D_a\bigl(b^n\xi(n)\bigr)=b^n\mc D_{ab^{-1}}\bigl(\xi(n)\bigr).
\end{equation}

Further, for $l,n\in\NNN$ with $n\ge Sl$, where $S$ is given by~\eqref{eq-S}, put
\begin{equation}\label{eq-vln}
v_{l,n}=v_{l,n}(\vec x)=\prod_{k=1}^l\prod_{j=1}^m\mc D_{\alpha_jq^{-k}}^{s_j}\bigl(v_n(\vec x)\bigr):=\left(\prod_{k=1}^l\prod_{j=1}^m\mc D_{\alpha_jq^{-k}}^{s_j}\right)\bigl(v_n(\vec x)\bigr)\in\CC[\vec x].
\end{equation}

Finally, let
\begin{equation}\label{eq-eps0}
\eps_0=\begin{cases}1,&\text{$P(z)=p_dz^d$, $p_d\in\QQ^*$,}\\0&\text{otherwise.}\end{cases}
\end{equation}

\begin{lemma}\label{lem-estimate-induction}
Let $l\ge d$, $\vec\omega=(\omega_0,\omega_{j,k,\sigma})\in\CC^{1+dS}$. Assume that for $0\le\nu<l$ and $n\ge S\nu$ we have
\begin{equation*}
|v_{\nu,n}(\vec\omega)|\le|q|^{-\nu n+(S-\eps_0/d)\nu^2/2+an+b},
\end{equation*}
where $a>0$ and $b$ don't depend on~$\nu$ and $n$. Then for $n\ge Sl$ we have
\begin{equation*}
|v_{l,n}(\vec\omega)|\le|q|^{-ln+(S-\eps_0/d)l^2/2+an+b+a+c'},
\end{equation*}
where $c'$ is a positive constant depending only on~$q,P,m,\alpha_j,s_j$.
\end{lemma}

\begin{proof}
Since $l\ge d$, it follows from~\eqref{eq-un} and~\eqref{eq-linear-recurrence} that
\begin{equation*}
\prod_{k=1}^l\prod_{j=1}^m\mc D_{\alpha_jq^{d-k}}^{s_j}\bigl(u_n(\vec\omega)\bigr)=0\qquad(n\in\ZZ).
\end{equation*}
Therefore, from~\eqref{eq-vn-linear-recurrence} we have
\begin{equation}\label{eq-intermediate-relation}
\prod_{k=1}^l\prod_{j=1}^m\mc D_{\alpha_jq^{d-k}}^{s_j}\left(v_{n+1}(\vec\omega)-P(q^{n+1})v_n(\vec\omega)\right)=0\qquad(n\ge Sl).
\end{equation}

Let
\begin{equation*}
P(z)=\sum_{\nu=0}^dp_\nu z^\nu.
\end{equation*}
Then in view of~\eqref{eq-mcD-identity} the relation~\eqref{eq-intermediate-relation} can be rewritten in the form
\begin{equation}\label{eq-main-relation}
p_dv_{l,n}(\vec\omega)=q^{-d(n+1)}\prod_{k=1}^l\prod_{j=1}^m\mc D_{\alpha_jq^{d-k}}^{s_j}\bigl(v_{n+1}(\vec\omega)\bigr)-\sum_{\nu=1}^dp_{d-\nu}q^{-\nu(n+1)}\prod_{k=1}^l\prod_{j=1}^m\mc D_{\alpha_jq^{\nu-k}}^{s_j}\bigl(v_n(\vec\omega)\bigr).
\end{equation}

It follows from the conditions of the lemma and~\eqref{eq-mcD-identity} that for $1\le\nu\le d$ we have
\begin{multline}\label{eq-main-inequality}
\left|q^{-\nu n}\prod_{k=1}^l\prod_{j=1}^m\mc D_{\alpha_jq^{\nu-k}}^{s_j}\bigl(v_{n+\eps_0}(\vec\omega)\bigr)\right|=\left|q^{-\nu n}\prod_{k=0}^{\nu-1}\prod_{j=1}^m\mc D_{\alpha_jq^k}^{s_j}\bigl(v_{l-\nu,n+\eps_0}(\vec\omega)\bigr)\right|\le\\
\le|q|^{-\nu n}\prod_{k=0}^{\nu-1}\prod_{j=1}^m\mc D_{-|\alpha_jq^k|}^{s_j}\left(|q|^{-(l-\nu)(n+\eps_0)+(S-\eps_0/d)(l-\nu)^2/2+a(n+\eps_0)+b}\right)=\\
=|q|^{-\nu n-(l-\nu)(n+\eps_0)+(S-\eps_0/d)(l-\nu)^2/2+a(n+\eps_0)+b}\prod_{k=0}^{\nu-1}\prod_{j=1}^m\bigl(1+|\alpha_jq^{k+l-\nu-a}|\bigr)^{s_j}\le\\
\le|q|^{-ln+(S-\eps_0/d)l^2/2-(1-\nu/d)\eps_0l+a(n+\eps_0)+b+c_1}\le|q|^{-ln+(S-\eps_0/d)l^2/2+a(n+\eps_0)+b+c_1},
\end{multline}
where $c_1$ is a constant depending only on~$q,P,m,\alpha_j,s_j$.

The lemma follows from~\eqref{eq-main-relation} and~\eqref{eq-main-inequality}.
\end{proof}

\begin{lemma}\label{lem-estimate}
Let $\vec\omega=(\omega_0,\omega_{j,k,\sigma})\in\CC^{1+dS}$ be such that
\begin{equation*}
\omega_0+\sumjks\omega_{j,k,\sigma}f^{(\sigma)}(\alpha_jq^k)=0.
\end{equation*}
Then for $l\ge0$ and $n\ge Sl$ we have
\begin{equation*}
|v_{l,n}(\vec\omega)|\le\max_{j,k,\sigma}|\omega_{j,k,\sigma}|\cdot|q|^{-ln+(S-\eps_0/d)l^2/2+c(n+1)},
\end{equation*}
where $c$ is a positive constant depending only on~$q,P,m,\alpha_j,s_j$.
\end{lemma}

\begin{proof}
In the proof we denote by~$c_1,c_2,c_3$ positive constants depending only on~$q,P,m,\alpha_j,s_j$.

It follows from~\eqref{eq-un} that
\begin{equation*}
\omega_0+\sum_{n=0}^\infty\frac{u_n(\vec\omega)}{\prod_{k=1}^nP(q^k)}=\omega_0+\sumjks\omega_{j,k,\sigma}f^{(\sigma)}(\alpha_jq^k)=0.
\end{equation*}
Hence \eqref{eq-vn} gives
\begin{equation}\label{eq-intermediate-relation-2}
v_n(\vec\omega)=-\sum_{l=n+1}^\infty\frac{u_l(\vec\omega)}{\prod_{k=n+1}^lP(q^k)}.
\end{equation}

It follows from~\eqref{eq-un} that for $n\ge1$ we have
\begin{equation*}
|u_n(\vec\omega)|\le c_1^n\max_{j,k,\sigma}|\omega_{j,k,\sigma}|.
\end{equation*}
Hence \eqref{eq-intermediate-relation-2} gives
\begin{equation*}
|v_n(\vec\omega)|\le\max_{j,k,\sigma}|\omega_{j,k,\sigma}|\cdot\sum_{l=n+1}^\infty\frac{c_2c_1^l}{(2c_1)^{l-n}}=c_2c_1^n\max_{j,k,\sigma}|\omega_{j,k,\sigma}|.
\end{equation*}
Consequently for $0\le\nu<d$ and $n\ge S\nu$ we have
\begin{equation*}
|v_{\nu,n}(\vec\omega)|\le\max_{j,k,\sigma}|\omega_{j,k,\sigma}|\cdot|q|^{c_3(n+1)}\le\max_{j,k,\sigma}|\omega_{j,k,\sigma}|\cdot|q|^{-\nu n+(S-\eps_0/d)\nu^2/2+(c_3+d)n+c_3}.
\end{equation*}

It follows from Lemma~\ref{lem-estimate-induction} that for $l\ge0$ and $n\ge Sl$ we have
\begin{equation*}
|v_{l,n}(\vec\omega)|\le\max_{j,k,\sigma}|\omega_{j,k,\sigma}|\cdot|q|^{-ln+(S-\eps_0/d)l^2/2+(c_3+d)n+c_3+(c_3+d+c')l},
\end{equation*}
where $c'$ is the constant of Lemma~\ref{lem-estimate-induction}. Using $l\le n/S$, we obtain the lemma.
\end{proof}

\section{Non-vanishing lemma}

\begin{lemma}\label{lem-rationality}
Let $\vec\omega=(\omega_0,\omega_{j,k,\sigma})\in\CC^{1+dS}$ be such that for some $l_0,n_0\in\NNN$ with $n_0\ge Sl_0$ we have
\begin{equation}\label{eq-vln-vanish}
v_{l_0,n_0}(\vec\omega)=v_{l_0,n_0+1}(\vec\omega)=\ldots=v_{l_0,n_0+dS}(\vec\omega)=0.
\end{equation}
Then the generating function~$F(z)$ of the sequence~$v_n(\vec\omega)$,
\begin{equation*}
F(z)=\sum_{n=0}^\infty v_n(\vec\omega)z^n\in\CC[[z]],
\end{equation*}
is rational.
\end{lemma}

\begin{proof}
Consider the sequence~$\{w_n\}_{n\ge0}$ given by
\begin{gather*}
w_n=v_{n_0-Sl_0+n}(\vec\omega)\qquad(0\le n<Sl_0),\\
\prod_{k=1}^{l_0}\prod_{j=1}^m\mc D_{\alpha_jq^{-k}}^{s_j}(w_n)=0\qquad(n\ge Sl_0),
\end{gather*}
where $\mc D_a$ is given by~\eqref{eq-mcD}. From~\eqref{eq-vln} and~\eqref{eq-vln-vanish} it follows that
\begin{equation}\label{eq-wn-equals-vn}
w_n=v_{n_0-Sl_0+n}(\vec\omega)\qquad(0\le n\le S(l_0+d)).
\end{equation}

It follows from~\eqref{eq-mcD-identity} that for~$\nu\in\ZZ$ we have
\begin{equation*}
\prod_{k=1}^{l_0}\prod_{j=1}^m\mc D_{\alpha_jq^{\nu-k}}^{s_j}\bigl(q^{\nu n}w_n\bigr)=q^{\nu n}\prod_{k=1}^{l_0}\prod_{j=1}^m\mc D_{\alpha_jq^{-k}}^{s_j}(w_n)=0\qquad(n\ge Sl_0).
\end{equation*}
Hence the sequence
\begin{equation*}
z_n=w_{n+1}-P(q^{n_0-Sl_0+n+1})w_n-u_{n_0-Sl_0+n+1}(\vec\omega)\qquad(n\ge0)
\end{equation*}
satisfies the linear recurrence relation
\begin{equation*}
\prod_{k=-l_0}^{d-1}\prod_{j=1}^m\mc D_{\alpha_jq^k}^{s_j}(z_n)=0\qquad(n\ge S(l_0+d))
\end{equation*}
of order~$S(l_0+d)$.

On the other hand, it follows from~\eqref{eq-vn-linear-recurrence} and~\eqref{eq-wn-equals-vn} that $z_n=0$ for~$0\le n<S(l_0+d)$. Hence $w_n=v_{n_0-Sl_0+n}(\vec\omega)$ for all~$n\ge0$, i.\,e., $v_n(\vec\omega)$ is linear recurrent and
\begin{equation*}
F(z)=\sum_{n\ge0}v_n(\vec\omega)z^n\in\CC(z).
\end{equation*}
This completes the proof. \end{proof}

\begin{lemma}\label{lem-non-rationality}
Let $\alpha_1,\ldots,\alpha_m$ satisfy the conditions~\ref{condition-1}--\ref{condition-2} of Theorem~\ref{th-1}, $\vec\omega=(\omega_0,\omega_{j,k,\sigma})\in\CC^{1+dS}\setminus\{\vec0\}$. Then the generating function~$F(z)$ of the sequence~$v_n(\vec\omega)$,
\begin{equation*}
F(z)=\sum_{n=0}^\infty v_n(\vec\omega)z^n\in\CC[[z]],
\end{equation*}
is not rational.
\end{lemma}

\begin{proof}
Assume the converse. Then for some constant~$C>1$ we have $|v_n(\vec\omega)|=O(C^n)$. It follows from~\eqref{eq-un} and~\eqref{eq-vn} that
\begin{equation*}
\omega_0+\sumjks\omega_{j,k,\sigma}f^{(\sigma)}(\alpha_jq^k)=\omega_0+\sum_{n=0}^\infty\frac{u_n(\vec\omega)}{\prod_{k=1}^nP(q^k)}=0.
\end{equation*}
In particular, not all~$\omega_{j,k,\sigma}$ vanish.

From~\eqref{eq-vn-linear-recurrence} it follows that $F(z)$ satisfies the functional equation
\begin{equation}\label{eq-functional-equation}
(1-p_0z)F(z)=\sum_{\nu=1}^dp_\nu q^\nu zF(q^\nu z)+R(z),
\end{equation}
where
\begin{gather*}
P(z)=\sum_{\nu=0}^dp_\nu z^\nu,\\
R(z)=\omega_0+\sum_{n=0}^\infty u_n(\vec\omega)z^n=\omega_0+\sumjks\frac{\omega_{j,k,\sigma}\sigma!z^\sigma}{(1-\alpha_jq^kz)^{\sigma+1}}\in\CC(z).
\end{gather*}

The condition~\ref{condition-1} of Theorem~\ref{th-1} implies that all $\alpha_jq^k$ are different. Since not all $\omega_{j,k,\sigma}$ vanish, the function~$R(z)$ has at least one pole. It follows from~\eqref{eq-functional-equation} that $F(z)$ also has a pole in~$\CC^*$.

We claim that any pole of~$F(z)$ is of the form~$\alpha_j^{-1}q^n$ with~$n\in\NN$. Assume the contrary. Let $\beta$ be a pole that cannot be represented in this form with the least~$|\beta|$. Then $R(z)$ doesn't have a pole at the point~$\beta q^{-d}$. It follows from~\eqref{eq-functional-equation} that one of the functions~$F(q^\nu z)$ with $0\le\nu<d$ has a pole at $\beta q^{-d}$. Hence we have $\beta=\beta'q^{d-\nu}$ for some pole~$\beta'$ of~$F(z)$. But then $|\beta'|<|\beta|$. Consequently $\beta'$ can be represented in the required form as well as~$\beta$. This contradiction proves our claim about poles of~$F(z)$. In particular, it follows from the condition~\ref{condition-1} of Theorem~\ref{th-1} that $F(z)$ and $R(z)$ do not have common poles.

Now suppose $\beta$ is a pole of~$F(z)$ with maximal~$|\beta|$. It follows from~\eqref{eq-functional-equation} and the above that the function~$(1-p_0z)F(z)$ does not have a singularity at the point~$\beta$. Hence $p_0\beta=1$. Since $\beta=\alpha_j^{-1}q^n$ with $n\in\NN$, this contradicts the condition~\ref{condition-2} of Theorem~\ref{th-1}. This contradiction proves the lemma.
\end{proof}

From Lemmas~\ref{lem-rationality} and~\ref{lem-non-rationality}, we get the following non-vanishing lemma.

\begin{lemma}
Let $\alpha_1,\ldots,\alpha_m$ satisfy the conditions~\ref{condition-1}--\ref{condition-2} of Theorem~\ref{th-1}, $\vec\omega=(\omega_0,\omega_{j,k,\sigma})\in\CC^{1+dS}\setminus\{\vec0\}$. Then for any $l_0,n_0\in\NNN$ with $n_0\ge Sl_0$ there exists an integer~$n$ with $n_0\le n\le n_0+dS$ such that $v_{l_0,n}(\vec\omega)\ne0$.\qed
\end{lemma}

\section{Main proposition}

Suppose~$\alpha_j\in\QQ^*$ ($1\le j\le m$). Denote by~$D$ any positive integer such that $DP(z)\in\ZZ[z]$ and $D\alpha_jq^k\in\ZZ$ for $1\le j\le m$, $0\le k<d$. For $l,n\in\NNN$ with $n\ge Sl$ consider
\begin{equation*}
w_{l,n}=w_{l,n}(\vec x)=D^nq_1^{Sl(l+1)/2}q_2^{dn(n+1)/2}v_{l,n}(\vec x).
\end{equation*}
It follows from~\eqref{eq-un} and~\eqref{eq-vn} that
\begin{equation*}
D^nq_2^{dn(n+1)/2}v_n\in\ZZ[\vec x]\qquad(n\ge0).
\end{equation*}
Combining this with~\eqref{eq-vln}, we get~$w_{l,n}\in\ZZ[\vec x]$.

For a linear form~$L$ denote by~$\mc H(L)$ the maximum of absolute values of its coefficients. From~\eqref{eq-un} and~\eqref{eq-vn} it follows that
\begin{equation*}
\mc H(v_n)\le|q|^{dn^2/2+O(n+1)}.
\end{equation*}
In view of~\eqref{eq-vln} the same estimate is valid for~$\mc H(v_{l,n})$ ($n\ge Sl\ge0$). Finally, for~$w_{l,n}$ we have
\begin{equation*}
\mc H(w_{l,n})\le|q_1|^{dn^2/2+Sl^2/2+O(n+1)}\qquad(n\ge Sl\ge0).
\end{equation*}

The above can be summarized as follows.

\begin{proposition}\label{prop-main}
Under the hypotheses of Theorem~\ref{th-1}, for any $l,n\in\NNN$ with $n\ge Sl$ there exists a linear form~$w_{l,n}=w_{l,n}(\vec x)\in\ZZ[\vec x]$ such that the following conditions hold:
\begin{enumerate}
\item $\mc H(w_{l,n})\le|q_1|^{dn^2/2+Sl^2/2+O(n+1)}$,
\item for any $\vec\omega=(\omega_0,\omega_{j,k,\sigma})\in\CC^{1+dS}$ such that
\begin{equation*}
\omega_0+\sumjks\omega_{j,k,\sigma}f^{(\sigma)}(\alpha_jq^k)=0
\end{equation*}
we have
\begin{equation*}
|w_{l,n}(\vec\omega)|\le\max_{j,k,\sigma}|\omega_{j,k,\sigma}|\cdot|q_1|^{\gamma dn^2/2-(1-\gamma)ln+\bigl((1-\gamma/2)S-(1-\gamma)\eps_0/(2d)\bigr)l^2+O(n+1)},
\end{equation*}
where $\gamma$ and $\eps_0$ are given by~\eqref{eq-gamma} and~\eqref{eq-eps0},
\item for any $\vec\omega=(\omega_0,\omega_{j,k,\sigma})\in\CC^{1+dS}\setminus\{\vec0\}$ and $l_0,n_0\in\NNN$ with $n_0\ge Sl_0$ there exists an integer~$n$ with $n_0\le n\le n_0+dS$ such that $w_{l_0,n}(\vec\omega)\ne0$.
\end{enumerate}
The constants in the Landau symbols~$O(\cdot)$ depend only on~$q,P,m,\alpha_j,s_j$.\qed
\end{proposition}

\section{Proof of Theorem~\ref{th-1}}

Take
\begin{equation*}
n_0=\left\lceil\frac{dS-\eps_0/2+\sqrt{(dS)^2+(1-\eps_0)dS+\eps_0^2/4}}d\,l\right\rceil=\left\lceil\frac{(M-1)l}d\right\rceil\ge Sl,
\end{equation*}
where $M$ is given by~\eqref{eq-M} and $l\in\NNN$ will be chosen later. It follows from Proposition~\ref{prop-main} that there exists an integer~$n=n_0+O(1)$ such that $w_{l,n}(\vec A)\ne0$. Since $w_{l,n}\in\ZZ[\vec x]$, we get
\begin{equation*}
|w_{l,n}(\vec A)|\ge1.
\end{equation*}

Let $\vec\omega=(\omega_0,\omega_{j,k,\sigma})$ be given by
\begin{gather*}
\omega_{j,k,\sigma}=A_{j,k,\sigma},\\
\omega_0=-\sumjks\omega_{j,k,\sigma}f^{(\sigma)}(\alpha_jq^k).
\end{gather*}
Using Proposition~\ref{prop-main}, we get
\begin{equation*}
|w_{l,n}(\vec\omega)|\le H|q_1|^{-al^2+O(l+1)},
\end{equation*}
where
\begin{equation*}
a=\frac{1-M\gamma}d\sqrt{(dS)^2+(1-\eps_0)dS+\eps_0^2/4}.
\end{equation*}
Take $l=(L/a)^{1/2}+O(1)$, where $L=\frac{\log H}{\log|q_1|}$, such that
\begin{equation*}
|w_{l,n}(\vec\omega)|\le1/2.
\end{equation*}
Then we have
\begin{equation*}
|w_{l,n}(\vec A)-w_{l,n}(\vec\omega)|\ge1/2.
\end{equation*}

On the other hand, using~Proposition~\ref{prop-main}, we get
\begin{equation*}
|w_{l,n}(\vec A)-w_{l,n}(\vec\omega)|\le\mc H(w_{l,n})|A_0-\omega_0|\le|A_0-\omega_0|\cdot|q_1|^{\mu L+O(L^{1/2})},
\end{equation*}
where $\mu$ is given by~\eqref{eq-mu}. Since
\begin{equation*}
|A_0-\omega_0|=\left|A_0+\sumjks A_{j,k,\sigma}f^{(\sigma)}(\alpha_jq^k)\right|,
\end{equation*}
we obtain Theorem~\ref{th-1}.

\end{document}